\documentclass[11pt,a4paper,draft]{amsart} 
\usepackage{amssymb,amscd,amsmath, young}
\usepackage{mathrsfs, mathdots} 
\usepackage[mathcal]{eucal} 
\usepackage{float}
\usepackage[usenames]{color}
\usepackage{soul}

\theoremstyle{plain}
\newtheorem{thm}{Theorem}[section]
\newtheorem{lem}[thm]{Lemma}

\newtheorem{cor}[thm]{Corollary}
\newtheorem*{claim*}{Claim}

\newtheorem*{con*}{Conjecture}
\newtheorem{lemma}[thm]{Lemma}

\theoremstyle{remark}
\newtheorem{rem}[thm]{Remark}

\newtheorem{exm}[thm]{Example}

\newtheorem{dfn}[thm]{Definition}
\newtheorem*{acknowledgements}{Acknowledgements}

\numberwithin{equation}{section}
\numberwithin{table}{section}

\newcommand{\wh}{\widehat}

\newcommand{\N}{\mathbb{N}}
\newcommand{\Z}{\mathbb{Z}}
\newcommand{\Q}{\mathbb{Q}}
\newcommand{\F}{\mathbb{F}}

\newcommand{\mfp}{\mathfrak{p}}

\newcommand{\ol}{\overline}
\newcommand{\olbfg}{\overline{{\bf g}}}

\newcommand{\Gri}{\ensuremath{\mathcal{O}}}

\renewcommand{\epsilon}{\varepsilon}

\renewcommand{\phi}{\varphi}
\renewcommand{\theta}{\vartheta}
\newcommand{\balpha}{{\alpha}}
\newcommand{\mcO}{\mathcal{O}}

\newcommand{\Qp}{\mathbb{Q}_p}

\newcommand{\rarr}{\rightarrow}
\newcommand{\epsi}{\varepsilon}
\newcommand{\val}{\mathrm{val}}

\newcommand{\zideal}{\zeta^{\triangleleft}}
\newcommand{\ceilalpha}{\lceil \balpha \rceil}

\DeclareMathOperator{\len}{len}

\DeclareMathOperator{\Des}{Des}
\DeclareMathOperator{\Gr}{Gr}

\DeclareMathOperator{\SL}{SL}
\DeclareMathOperator{\GL}{GL}

\DeclareMathOperator{\Mat}{Mat}

\def \mfm {\mathfrak{m}}

\def \bfo {{\bf 1}}
\def \bfG {{\bf G}}

\def \bfe {{\bf e}}
\def \bff {{\bf f}}

\def \bfh {{\bf h}}

\def \bfr {{\bf r}}

\def \bfY {{\bf Y}}

\def \Fp {\ensuremath{\mathbb{F}_p}}

\def \mcR {\ensuremath{\mathcal{R}}}

\def \p {\ensuremath{\mathfrak{p}}}

\def \Zp  {\mathbb{Z}_p}

\author{Michael M.~Schein} \address{Department of Mathematics,
  Bar-Ilan University, Ramat Gan 52900,
  Israel}\email{mschein@math.biu.ac.il}

\author{Christopher Voll} \address{Fakult\"at f\"ur Mathematik,
  Universit\"at Bielefeld, D-33501 Bielefeld, Germany}
\email{C.Voll.98@cantab.net}



\thanks{Schein was supported by grant 2264/2010 from the
  Germany-Israel Foundation for Scientific Research and Development
  and a grant from the Pollack Family Foundation.  We acknowledge
  support by the DFG Sonderforschungsbereich 701 ``Spectral Structures
  and Topological Methods in Mathematics'' at Bielefeld University.}

\begin{document}
 \title[Normal zeta functions of Heisenberg groups over number rings
   II]{Normal zeta functions of the Heisenberg groups over number
   rings II - the non-split case} \date{\today}

 \begin{abstract} 
 We compute explicitly the normal zeta functions of the Heisenberg
 groups $H(R)$, where $R$ is a compact discrete valuation ring of
 characteristic zero. These zeta functions occur as Euler factors of
 normal zeta functions of Heisenberg groups of the form $H(\mcO_K)$,
 where $\mcO_K$ is the ring of integers of an arbitrary number
 field~$K$, at the rational primes which are
 non-split in~$K$. We show that these local zeta functions satisfy
 functional equations upon inversion of the prime.
\end{abstract}
 \maketitle

\thispagestyle{empty}

\section{Introduction}
Let $G$ be a finitely generated abstract or profinite group. For
$m\in\N$, let $a^\triangleleft_m(G)$ denote the number of (open)
normal subgroups of $G$ of index $m$ in~$G$. The \emph{normal zeta
  function} of $G$ is the Dirichlet generating series
$$\zideal_G(s) = \sum_{m=1}^\infty a^\triangleleft_m(G)m^{-s},$$ where
$s$ is a complex variable. If $G$ is a finitely generated nilpotent
group, then its normal zeta function converges on a complex half-plane
and satisfies the Euler product
\begin{equation*}
\zideal_G(s) = \prod_{p \textrm{ prime}}\zideal_{G,p}(s).
\end{equation*}
Here, for a prime $p$, the Euler factor $\zideal_{G,p}(s) =
\sum_{k=0}^\infty a^\triangleleft_{p^k}(G)p^{-ks}$ enumerates the
normal subgroups of $G$ of $p$-power index in $G$. It may also be
viewed as the normal zeta function of the pro-$p$ completion
$\wh{G}^p$ of $G$. The Euler product reflects the facts that the
normal zeta function of $G$ coincides with the normal zeta function of
its profinite completion~$\wh{G}$ and that $\wh{G} \cong \prod_{p
  \textrm{ prime}}\wh{G}^p$. The zeta functions
$\zideal_{G,p}(s)$ are known to be rational functions in $p^{-s}$;
cf.~\cite[Theorem~1]{GSS/88}.

Given a ring $\mcR$, the \emph{Heisenberg group} $H(\mcR)$ over $\mcR$
is the group of upper unitriangular $3\times 3$ matrices over $\mcR$:\
$$H(\mcR) = \left\{\left( \begin{matrix}
  1&a&c\\0&1&b\\0&0&1\end{matrix}\right) \mid a,b,c\in \mcR
  \right\}.$$ If $\mcR$ is a finitely generated torsion-free
  $\Z$-module of rank $n$, say, then $H(\mcR)$ is a finitely
  generated torsion-free nilpotent group of nilpotency class $2$ and
  Hirsch length $3n$. Given a prime $p$, the pro-$p$ completion of
  $H(\mcR)$ is isomorphic to the $3n$-dimensional nilpotent
  $p$-adic analytic pro-$p$ group $H(\mcR_p)$, where
  $\mcR_p = \mcR \otimes_\Z \Zp$, and we have
$$\zideal_{H(\mcR),p} = \zideal_{H(\mcR_p)}.$$

In this article we compute an explicit formula for the normal zeta
function of the Heisenberg group over an arbitrary compact discrete
valuation ring $R$ of characteristic zero, i.e.\ a finite extension of
the ring $\Zp$ of $p$-adic integers. Let $\mfm$ be the maximal ideal
of the local ring~$R$. The residue field $k_R =R/\mfm$ is a finite
extension of the prime field~$\Fp$.  Its degree $f=[k_R :\Fp]$ is
called the \emph{inertia degree} of $R$. The (\emph{absolute})
\emph{ramification index} $e$ of $R$ is given by $pR=\mfm^e$. The ring
$R$ is called \emph{unramified} (over $\Zp$) if $e=1$ and
\emph{totally ramified} (over $\Zp$) if~$f=1$. The \emph{degree} of
$R$ as an extension of $\Zp$ is $n=ef$. It coincides with the rank of
$R$ as a $\Zp$-module.

Normal zeta functions of Heisenberg groups of the form $H(R)$ occur as
Euler factors of normal zeta functions of Heisenberg groups over
number rings. Indeed, let $\mcO_K$ be the ring of integers of a number
field~$K$. Then $(\mcO_K)_p=\mcO_K\otimes_\Z\Zp$ is a local ring
precisely if $p$ does not split in $K$, i.e.\ it decomposes in $K$ as
$p\mcO_K = \mfp^e$, where $\mfp$ is a prime ideal of~$\mcO_K$.  
In this case, $f = [\mathcal{O}_K / \p : \F_p]$ and $n = ef = [K : \Q]$ is the degree of $K$.
We call such primes \emph{non-split} (in~$K$).  Note that all finite extensions of $\Zp$ arise in
this way.

It follows from the general result~\cite[Theorem~1]{GSS/88} that
normal zeta functions of groups of the form $H(R)$ are rational in
$p^{-s}$. The more specific result \cite[Theorem~3]{GSS/88} asserts
that the Euler factors of (normal) zeta functions of $H(\mcO_K)$ are
rational in the two parameters $p^{-s}$ and $p$ on sets of rational
primes with fixed decomposition type in $K$; cf.\ also \cite{SV1/13}
for details. There are, in particular, rational functions
$W^\triangleleft_{e,f}(X,Y)\in\Q(X,Y)$ such that for all rational
primes~$p$ and rings $R$ as above, the following holds:
$$\zideal_{H(R)}(s) = W^\triangleleft_{e,f}(p,p^{-s}).$$
 
In Theorem~\ref{thm:main}, our main result, we compute the rational
functions $W^\triangleleft_{e,f}(X,Y)$ explicitly. Moreover, we
prove the following functional equation in Corollary~\ref{cor:funeq}.

\begin{thm}\label{thm:main.funeq}
Let $e,f\in\N$ with $ef=n$. Then
\begin{equation} \nonumber \label{equ:funeq.intro}
W^\triangleleft_{e,f}(X^{-1},Y^{-1}) = (-1)^{3n}
X^{\binom{3n}{2}}Y^{5n+2(e-1)f}W^\triangleleft_{e,f}(X,Y).
\end{equation}
\end{thm}

Note that $3n = \dim(H(R))$ and $5n = \dim(H(R)) + \dim(H(R)/H(R)')$,
where $H(R)'$ is the derived subgroup of $H(R)$.  Here ``$\dim$''
refers to the dimensions as $p$-adic analytic pro-$p$ groups. The term
$2(e-1)f$ in the exponent of $Y$ describes the deviation from the ``generic'' symmetry factor
in the functional equations for the local factors of normal zeta
functions of finitely generated nilpotent groups of nilpotency
class~$2$; cf.~\cite[Theorem~C]{Voll/10}.

Prior to our work, the normal zeta functions $\zideal_{H(R)}$ had been
calculated for all cases occurring for~$n\leq 3$; see \cite[Theorems
  2.3, 2.7, and 2.9]{duSWoodward/08}.

\subsection{Methodology}
The results of the current paper complement those of~\cite{SV1/13},
where we carry out analogous computations of the normal zeta functions
of the groups $H((\mcO_K)_p)$ for primes $p$ which are unramified
in the number field~$K$. In \cite[Theorem~1.2]{SV1/13} we establish functional
equations for these zeta functions that are comparable to those in
Theorem~\ref{thm:main.funeq}. Our results agree, of course, in the
common special case of primes $p$ which are inert in $K$; see
Theorem~\ref{thm:inert}. In \cite[Conjecture~1.4]{SV1/13} we
conjecture a functional equation for $\zideal_{H(\mcO_K),p}(s)$ for
arbitrary (not necessarily unramified or non-split) primes.

The methods used in the present paper are, however, quite different
from those of~\cite{SV1/13}. There the problem of computing the
relevant zeta functions reduces to that of effectively enumerating
subgroups of finite abelian $p$-groups varying in infinite,
combinatorially described families. The precise shape these families
may take is determined by the decomposition type of the rational prime
$p$ in the number field~$K$. The sum defining the local zeta function
is organized as a finite sum, indexed by certain Dyck words.

The decomposition type that leads to the combinatorially simplest
situation is that of inert primes, namely the case where $p
\mathcal{O}_K$ is a prime ideal.  We view the non-split case considered in this paper as
a degeneration of the inert case and tackle it using geometric and
Coxeter-group-theoretic ideas introduced in \cite{Voll/05}
and~\cite{KlopschVoll/09}, as we now explain.

The paper \cite{Voll/05} argues that the normal subgroup growth of a
finitely generated nilpotent group $G$ of nilpotency class $2$ is,
to a large extent, determined by the geometry of its \emph{Pfaffian
  hypersurface}. This is a projective hypersurface, defined explicitly
by the Pfaffian of an antisymmetric matrix of linear forms encoding
the group's structure constants with respect to a chosen (Mal'cev)
basis. If the Pfaffian hypersurface of $G$ is smooth and
contains no lines, and $G$ satisfies some other mild hypotheses, then
\cite[Theorem~3]{Voll/05} gives an explicit formula for the Euler
factors $\zideal_{G,p}$, at almost all primes~$p$, in terms of
the numbers of $\Fp$-rational points on the Pfaffian
hypersurface. This formula presents the Euler factor as the sum of an
\emph{approximative term}, which coincides with the Euler factor if
and only if the Pfaffian hypersurface has no $\Fp$-rational point, and
a \emph{correction term}, which corrects the approximation along the
hypersurface's $\Fp$-points. In the special case $G=H(\mcO_K)$, the
results of \cite{Voll/05} are not directly applicable.  The Pfaffian
hypersurface is the union of $n$ hyperplanes in general position in 
$(n-1)$-dimensional projective space. It has no $\Q$-rational points;
over $\Qp$ it splits as a union of (restrictions of scalars of)
 hyperplanes, in a way determined by the decomposition
behaviour of $p$ in~$K$. The computation of the relevant local zeta
function comes down to a detailed quantitative analysis of the
interplay between these fixed hyperplanes and varying $p$-adic
lattices.

In the case of inert primes $p$, the ideas of \cite{Voll/05} do apply directly to the Euler factors $\zideal_{H(\Gri_K),p}$, as in this case
the Pfaffian hypersurface has no $\Fp$-rational points.  Thus the Euler
factor is equal to the approximative term mentioned above. In the
setup of~\cite{Voll/05}, this means that the set of solutions of a
certain system of linear congruences has a particularly simple
form. For non-split primes, ramification complicates this system only
slightly.  The main idea of the current article is to control this
complication using parabolic length functions on symmetric
groups. These functions generalize the usual Coxeter length and were
used to solve related enumeration problems
in~\cite{KlopschVoll/09}. Theorem~\ref{thm:main} expresses
$\zideal_{H(R)}(s)$ in terms of parabolic length functions on the
symmetric group $S_n$, whereas Corollary~\ref{cor:tot.ram} gives a
formula in the totally ramified case in terms of parabolic length
functions on~$S_{n-1}$. The functional equation expressed in
Theorem~\ref{thm:main.funeq} reflects the good behaviour of the relevant parabolic
length functions under (left-)multiplication by the Coxeter group's longest element.

\subsection{Outlook}
In this section we briefly describe some directions for future research building on the methods of the present paper and of~\cite{SV1/13}.

It would be of great interest to match the geometric setup of
\cite{Voll/05} precisely with the combinatorial approach taken in
\cite{SV1/13}, for the Heisenberg groups $H(\mathcal{O}_K)$ and also more generally. It is plausible that the presence of lines and higher-dimensional linear spaces on the Pfaffian
hypersurface necessitates further \emph{correction terms}, accounting
for the possible intersection types of flags with coordinate
hyperplanes.  
We note that Dyck words and possible intersection behaviours of a flag with a fixed set of hyperplanes in general position are both enumerated by the Catalan numbers; moreover, there is a natural bijection between these two types of objects.  For the case of $[K:\Q] = 3$, the correction terms arising from a generalization of the approach of \cite{Voll/05} appear to coincide with the functions associated to Dyck words that were computed in \cite{SV1/13}.  This is likely to be a special case of a very general phenomenon.


Let $g\in\N$. Given $g$-tuples $\mathbf{e} = (e_1,
\dots, e_g)\in\N^g$ and $\mathbf{f} = (f_1, \dots, f_g)\in\N^g$
satisfying $\sum_{i = 1}^g e_i f_i = [K:\Q]$, we say that a (rational)
prime $p$ is of \emph{decomposition type} $(\bfe,\bff)$ in the number field $K$
if 
\begin{equation*}
p \mathcal{O}_K = \p_1^{e_1} \cdots \p_g^{e_g},
\end{equation*}
where the $\p_i$ are distinct prime ideals in $\mcO_K$ with
ramification indices $e_i$ and inertia degrees $f_i = [\mathcal{O}_K /
  \p_i : \F_p]$ for~$i=1,\dots,g$.  We call the decomposition type $(\bfe, \bff)$ {\emph{unramified}} if $\bfe = \bfo = (1, \dots, 1)$.

  Taken together, \cite{SV1/13} and the present paper give explicit
  formulae for all but finitely many Euler factors of the global ideal
  zeta functions $\zideal_{H(\mcO_K)}(s)$.  Still outstanding is an
  analysis of the general ramified decomposition types.  In view of
  the geometric picture sketched above, it is suggestive to view a
  general decomposition type $(\bfe, \bff)$ as a degeneration of an
  associated unramified decomposition type $(\bfo, \bff^\prime)$,
  where $\bff^\prime = (e_1 f_1, \dots, e_g f_g)$.  The methods of
  this paper suggest trying to describe the effect of this
  degeneration on the zeta function, computed in~\cite{SV1/13}, of the
  unramified type $(\bfo, \bff^\prime)$ by means of suitable parabolic
  length functions or similar combinatorially described functions. The
  current paper carries out this idea for $g=1$.

  We see this paper and \cite{SV1/13} as first steps in a
  systematic study of the behaviour of (normal) subgroup growth of
  general nilpotent groups under base extension. Specifically, one may
  ask the following: given a
  finitely generated nilpotent group of the form $G = \bfG(\Z)$, arising as the group of
  $\Z$-rational points of a unipotent group scheme $\bfG$ defined over
  $\Z$, how does the normal subgroup growth sequence
  $(a^\triangleleft_m(\bfG(\Gri)))_{m\in\N}$ vary as $\Gri$ ranges
  over the rings of integers of number fields?  For instance, it seems
  reasonable to expect that the local factors of the associated normal
  zeta functions should admit some kind of uniform description on sets of
  (rational) primes of fixed decomposition type.

  The same expectation holds for zeta
  functions encoding other data, such as the subgroup growth sequence
  $(a_m(\bfG(\Gri)))_{m\in\N}$ counting \emph{all} finite index
  subgroups of $\bfG(\Gri)$. The associated Dirichlet series
  $\zeta_{\bfG(\Gri)}(s)$ are known to have Euler decompositions
  analogous to those of $\zeta^\vartriangleleft_{\bfG(\Gri)}(s)$.  It
  is very natural to
  try to extend the methodology developed in this paper and
  in~\cite{SV1/13} to the subgroup zeta factors
  $\zeta_{\bfG(\Gri),p}(s)$.  For the Heisenberg group, it is conjectured in \cite[p.~188]{GSS/88} that for
  every decomposition type $(\bfe, \bff)$ there exists a rational
  function $W_{\bfe,\bff}(X,Y)\in\Q(X,Y)$ such that for all rational
  primes $p$ of decomposition type $(\bfe, \bff)$ in $K$ the following
  holds:
$$\zeta_{H(\Gri_K),p}(s) = W_{\bfe,\bff}(p,p^{-s}).$$ 
While the analogous statement for normal zeta functions was already proved in~\cite{GSS/88}, to our
knowledge this conjecture has not even been completely settled for
$[K:\Q]=2$ (but see \cite[Theorem~2.4]{duSWoodward/08} for the case of
split primes). That counting all finite index subgroups is a
far more complex task than counting normal such subgroups is
reflected in the fact that systems of \emph{quadratic} Diophantine
equations take the role played by the systems of linear such
equations that we work with in this paper and in~\cite{SV1/13}.


\begin{acknowledgements}
We are grateful to Mark Berman for bringing us together to work on
this project and to the referee for helpful comments.
\end{acknowledgements}

\section{Preliminaries}
Let $p$ be a rational prime.  For an integer $m \geq 1$, we write $[m]$ for $\{ 1, 2,
\dots, m \}$ and $[m]_0$ for $\{ 0, 1, \dots, m \}$. Given integers
$a,b$ with $a\leq b$, we write $[a,b]$ for $\{a,a+1,\dots,b\}$. Given
a finite set $I$ of integers, we write $I=\{i_1,\dots,i_\ell\}_<$ to
indicate that $i_1<\dots<i_\ell$.

\subsection{Coxeter groups}\label{subsec:cox}
The symmetric group $S_n$ of degree $n$ is a Coxeter group with
Coxeter generating set $\mathcal{S}=\{s_1,\dots,s_{n-1}\}$, where, for
each $i\in[n-1]$, we denote by $s_i = (i \mbox{ } i + 1)$ the
transposition of the letters $i$ and $i+1$ in the standard permutation
representation of~$S_n$. We will frequently identify elements of $S_n$
with permutations of $[n]$ in this way.

We write $\len:S_n \rarr [\binom{n}{2}]_0$ for the usual Coxeter
length function: for $w\in S_n$, $\len(w)$ denotes the length of a
shortest word representing $w$ as a product of elements
of~$\mathcal{S}$.

Given $I \subseteq [n - 1]$, we write $W_I = \langle s_i \mid i \in I
\rangle$ for the parabolic subgroup of $S_n$ generated by the elements
of $\mathcal{S}$ indexed by elements of $I$. The restriction of $\len
$ to $W_I$ coincides with the standard length function on the Coxeter
group $W_I$.  Every element $w\in S_n$ can be factorized uniquely as
$w=w^Iw_I$, where $w_I\in W_I$ and $w^I$ is the unique element of
shortest length in the coset $wW_I$.  Moreover, $\len(w) = \len(w_I) +
\len(w^I)$; cf.~\cite[Section~1.10]{Humphreys/90}. We set $\len^{I}(w)
:= \len(w^I)$, and call $\len^I$ the (\emph{right}) \emph{parabolic
  length function} associated to $I$;
cf.~\cite[Definition~2.2]{KlopschVoll/09}.

The group $S_n$ has a unique longest element $w_0$ with respect to
$\len$, namely the inversion $w_0(i)=n+1-i$ for $i\in[n]$. Parabolic
length functions are well-behaved with respect to (left)
multiplication with $w_0$: for every $I\subseteq [n-1]$ and
$w\in S_n$,
\begin{equation}\label{equ:len.par}
\len^I(w_0w) = \len^I(w_0) - \len^{I}(w);
\end{equation}
cf.~\cite[Lemma~2.3]{KlopschVoll/09}. Clearly $\len =
\len^{\varnothing}$. The other parabolic length function relevant for
us is~$\len^{[n-2]}$. It is easy to check that $\len^{[n-2]}(w)=
n-w(n)$ for all $w \in S_n$, and in particular that $\len^{[n-2]}(w_0)=n-1$.

 The (\emph{right}) \emph{descent set} $\Des(w)$ of an element $w\in
 S_n$ is defined as
 $$\Des(w) = \{i\in[n-1] \mid \len(ws_i) < \len(w) \}.$$
 It is easily seen that $\Des(w) = \{i\in [n-1] \mid w(i+1) < w(i)\}$ and
\begin{equation}\label{equ:des.w0}
\Des(w_0w) = [n-1] \setminus \Des(w).
\end{equation}
\begin{exm} Consider the element
  $w\in S_6$ corresponding to the permutation matrix
   $$w =
  \left( \begin{matrix}0&0&0&1&0&0\\1&0&0&0&0&0\\0&0&0&0&1&0\\0&0&0&0&0&1\\0&1&0&0&0&0\\0&0&1&0&0&0 \end{matrix}\right).$$
  Here $\Des(w) = \{3\}$, $\len(w) = 7$ and $\len^{[4]}(w)= 6-w(6)=2$.
\end{exm}

For a variable $Y$ and integers
$a,b\in\N_0$ with $a\geq b$, the
\emph{Gaussian binomial coefficient} is defined to be
$$\binom{a}{b}_Y = \frac{\prod_{i=a-b+1}^a (1-Y^i)}{\prod_{i=1}^b
  (1-Y^i)}\in \Z[Y].$$
Given an integer $n\in\N$ and a subset $I  = \{ i_1, \dots, i_\ell \}_< \subseteq [n-1]$, the associated \emph{Gaussian
  multinomial} is defined as
$$ \binom{n}{I}_Y = \binom{n}{i_{\ell}}_Y \binom{i_{\ell}}{i_{\ell-1}}_Y \cdots
\binom{i_2}{i_1}_Y \in \Z [Y].$$
 Then (cf.~\cite[Section~1.7]{Stanley/12}) for $I\subseteq[n-1]$ we have
\begin{equation}\label{equ:des}
\sum_{w\in S_n,\; \Des(w) \subseteq I} Y^{\len(w)} = \binom{n}{I}_Y.
\end{equation}

\subsection{Grassmannians} \label{subsec:grass}
Given an integer $i\in[n]_0$, we denote by $\Gr(n,n-i)$ the
Grassmannian of $(n-i)$-dimensional subspaces of affine
$n$-dimensional space. This $i(n-i)$-dimensional projective variety
has a decomposition
$$\Gr(n,n-i) = \bigcup_{w\in S_n,\; \Des(w) \subseteq\{i\}} \Omega_w$$
into (\emph{Schubert}) \emph{cells} $\Omega_w$, indexed by
$\binom{n}{n-i}$ elements of~$S_n$. These cells have an elementary
realization as follows. Fix a vector space basis for affine
$n$-dimensional space. Subspaces of dimension $n-i$ may then be represented
by $\GL_{n-i}$-left cosets of matrices of size $n\times(n-i)$ of full
rank $n-i$. A set of such matrices of the form {
  \setlength{\arraycolsep}{8pt}
\begin{equation*} 
\left( \begin{matrix} * & * & \dots&& * \\ \vdots &&&& \vdots\\ * & *
  & \dots && *
  \\ 1&0&\dots&\dots&0\\ 0&*&\dots&\dots&*\\ \vdots&\vdots&&&\\ 0&1&0&\dots&0\\ 0&0&*&\dots&*\\ &\vdots&\vdots&&\vdots\\ 0&
  0&0&\dots&1\\ 0&&&&0\\ \vdots&&&&\vdots\\ 0&0&\dots&&0
 \end{matrix}\right)_{n \times (n-i)}
\end{equation*}}

\noindent 
where $*$ stands for arbitrary field elements, is a set of unique
coset representatives. More precisely, for any such matrix there is a
subset $J \subseteq [n]$ of cardinality $n - i$ such that the submatrix comprising rows labeled by
elements of $J$ is the $(n-i)$-identity matrix.  The matrix above has zeroes
in all entries below or to the right of a $1$ in this submatrix, and
arbitrary entries in the remaining positions.  The set of cosets
corresponding to such matrices for a fixed subset
$J=\{j_1,\dots,j_{n-i}\}_<\subseteq[n]$ may be identified with the
cell $\Omega_w$, where $w\in S_n$ is the unique element in $S_n$ whose
descent set is contained in $\{i\}$ and which satisfies $w(i + m) =
j_m$ for all $m \in [n-i]$. This illustrates that each cell $\Omega_w$
is an affine space of dimension $i(n-i) - \len(w)$, which is the number of symbols $*$
in the above matrix. Hence, given a prime $p$, the number
$\#\Gr(n,n-i;\Fp)$ of $\Fp$-rational points of $\Gr(n,n-i)$ is given
by the formula
\begin{equation}\label{equ:grass}
\#\Gr(n,n-i;\Fp)=\sum_{w\in S_n,\; \Des(w)
  \subseteq\{i\}}p^{i(n-i)-\len(w)} =
\binom{n}{n-i}_{p^{-1}}p^{i(n-i)} =\binom{n}{n-i}_{p};
\end{equation}
cf.~\eqref{equ:des}.  We refer to \cite[Section 3.2]{Manivel/01} for
further information about Schubert cells.

\subsection{Lattices}\label{subsec:lattices}
For the reader's convenience, we recall some notation used in \cite{Voll/05} to parameterize
sublattices $\Lambda\leq \Zp^n$. A sublattice $\Lambda\leq \Zp^n$ of finite index in
$\Zp^n$ is \emph{maximal in $\Zp^n$} if $p^{-1}\Lambda\not\leq
\Zp^n$. Such a lattice is called \emph{of type
  $\nu(\Lambda)=(I,\bfr_I)$}, where $I=\{i_1,\dots,i_\ell\}_< \subseteq
   [n-1]$ and $\bfr_I=(r_{i_1},\dots,r_{i_\ell})\in\N^\ell$, if $\Lambda$
   has elementary divisors
$$p^\nu :=
   (\underbrace{1,\dots,1}_{i_1},\underbrace{p^{r_{i_1}},\dots,p^{r_{i_1}}}_{i_2-i_1},\dots,\underbrace{p^{\sum_{\iota\in
         I}r_\iota},\dots,p^{\sum_{\iota\in I}r_\iota}}_{n-i_\ell})$$
   with respect to~$\Zp^n$. (Note that this ordering differs from the
   one used in \cite[Section~3.1]{Voll/10}.)
 
 Fix a $\Zp$-basis $(\epsilon_1,\dots,\epsilon_n)$
   of $\Zp^n$.
   The group $\Gamma = \SL_n(\Zp)$ acts transitively on the finite set
   of maximal sublattices of $\Zp^n$ of given type~$\nu =
   (I,\bfr_I)$.  Denote by $\Gamma_{(I,\bfr_I)}$ the stabilizer in
   $\Gamma$ of the diagonal lattice $\bigoplus_{j=1}^n (p^\nu)_j\Zp
   \epsilon_j$. This allows us to identify a given maximal lattice
   with a coset $\boldsymbol{\alpha} \Gamma_{(I,\bfr_I)}$, where
   $\boldsymbol{\alpha}\in\Gamma$. The number of maximal lattices of type
   $(I,\bfr_I)$ inside $\Zp^n$ is given by
\begin{equation}\label{equ:num.latt}
|\Gamma:\Gamma_{(I,\bfr_I)}| = \binom{n}{I}_{p^{-1}}p^{\sum_{\iota\in
    I}r_\iota \iota(n-\iota)};
\end{equation}
see, for instance, \cite[Eq.~(26)]{Voll/10}.

\subsection{Linearization} \label{sec:liering}
The problem of counting finite-index
normal subgroups of $H(R)$ turns out to be equivalent to the problem of counting
finite-index ideals in a certain Lie ring, which we now introduce.  Given a ring $\mcR$, the
\emph{Heisenberg Lie ring} $L(\mcR)$ over $\mcR$ is defined as
\begin{equation*} 
L(\mcR) = \left\{ \left( \begin{array}{ccc} 0 & a & c \\ 0 & 0 & b
  \\ 0 & 0 & 0 \end{array} \right) \mid a, b, c \in \mcR \right\},
\end{equation*} 
equipped with the Lie bracket induced from $\mathfrak{gl}_3(\mcR)$.
The derived subring $L(\mcR)^\prime$ of $L(\mcR)$ is equal to the
center of $L(\mcR)$ and consists of those matrices for which $a = b = 0$.  Let
$\overline{L(\mcR)} = L(\mcR) / L(\mcR)^\prime$ be the abelianization.

If $\mcR$ is an $A$-module of finite rank, for some commutative ring $A$, then so is
$L(\mcR)$. In this case, $L(\mcR)$ has only finitely many 
$A$-ideals of each finite index. The ($A$-)\emph{ideal zeta function} of
$L(\mcR)$ is then defined as the Dirichlet generating function
\begin{equation} \label{equ:def.ideal.zeta}
\zideal_{L(\mcR)}(s) = \sum_{n=1}^\infty
a_n^\triangleleft(L(\mcR))n^{-s},
\end{equation}
where $a_n^\triangleleft(L(\mcR))$ denotes the number of $A$-ideals of
index $n$ in $L(\mcR)$.  In the cases considered in this paper, we
have $A = \Z_p$.

\section{Computation of the functions $W^\triangleleft_{e,f}(X,Y)$}
\subsection{The set-up}
Let $R$ be a compact discrete valuation ring of characteristic zero,
with maximal ideal $\mathfrak{m}$ and finite residue field $k_R = R /
\mathfrak{m}$.  Fix a uniformizer $\pi \in \mathfrak{m}$, and let
$\val$ be the discrete valuation on $R$, normalized so that $\val
(\pi) = 1$.  Let $p$ be the characteristic of $k_R$ and $f = [k_R : \F_p]$
the inertia degree.  Denote by $e$ the ramification index of $R$,
which satisfies $pR = \mathfrak{m}^e$.  Note that there is a natural
ring embedding of $\Z_p$ into $R$, endowing $R$ with a $\Z_p$-module
structure.

Let $(\overline{\beta}_1, \dots, \overline{\beta}_f)$ be an ordered
$\F_p$-basis of $k_R$.  For each $i \in [f]$, we fix a lift $\beta_i \in
R$ of~$\overline{\beta}_i$.  Then $R$ is a free $\Z_p$-module of rank
$n = ef$ and the set
\begin{equation*} 
\mathcal{B} = \left\{ \beta_i \pi^j \mid i \in [f], j \in [e-1]_0 \right\}
\end{equation*}
is a $\Zp$-basis; cf.~\cite[Proposition II.6.8]{Neukirch/99}.  We
order it as follows: $\mathcal{B} = (d_1, \dots, d_n)$, where $d_{i +
  fj} = \beta_i \pi^j$.  Setting
$$
\begin{array}{lcr}
a_i = \left( \begin{array}{ccc} 0 & d_i & 0 \\ 0 & 0 & 0 \\ 0 & 0 &
  0 \end{array} \right), & a_{n + i} = \left( \begin{array}{ccc} 0 & 0
  & 0 \\ 0 & 0 & d_i \\ 0 & 0 & 0 \end{array} \right), & c_i =
\left( \begin{array}{ccc} 0 & 0 & d_i \\ 0 & 0 & 0 \\ 0 & 0 &
  0 \end{array} \right)
\end{array}
$$ for each $i \in [n]$, we obtain the following presentation of the
Heisenberg Lie ring $L(R)$ defined in Section \ref{sec:liering}:
\begin{equation} \label{equ:ring.pres}
  L(R) = \left \langle a_1, \dots, a_{2n}, c_1, \dots, c_n \mid [a_i,
    a_j] = M(\mathbf{c})_{ij},\; i,j \in [2n] \right \rangle.
\end{equation}

Here $M(\bfY) \in \Mat_{2n}(\Z_p [Y_1, \dots, Y_n])$ is a matrix whose
entries are $\Z_p$-linear forms in the variables $Y_1, \dots, Y_n$, and
$M(\mathbf{c})$ is the matrix obtained after making the substitution
$Y_i = c_i$ for all $i \in [n]$.  More precisely, $M(\mathbf{Y})$ has
the form
\begin{equation} \label{equ:mc}
M(\mathbf{Y}) = \left( \begin{array}{cc} 0 & B(\mathbf{Y}) \\ -
  B(\mathbf{Y}) & 0 \end{array} \right),
\end{equation}
where $B(\mathbf{Y})$ is an $n \times n$ matrix whose entries are
given by $B(\mathbf{Y})_{ij} = \sum_{k = 1}^n \gamma^{ij}_k Y_k$, and
the ``structure constants'' $\gamma^{ij}_k \in \Z_p$ are defined by
the relations $d_i d_j = \sum_{k=1}^n \gamma^{ij}_k d_k$.  Note that
\eqref{equ:ring.pres} differs from the presentation appearing in
\cite[Section 2.1]{SV1/13} by a reordering of the generators $a_i$.  

\begin{rem} \label{rem:structure.constants} Given integers $i, j
  \in [n]$, write $i = i_1 f + i_0$ and $j = j_1 f + j_0$, where $i_0,
  j_0 \in [f]$ and $i_1, j_1 \in [e-1]_0$.  Define $\eta = \pi^e / p
  \in \Z_p^\ast$.  It is immediate from the definition of the basis
  elements $d_i$ that $d_i d_j = \pi^{i_1 + j_1} d_{i_0} d_{j_0}$. We
  write $i_1 + j_1 = \ell_1 e + \ell_0$ for $\ell_0 \in [e-1]_0$ and
  $\ell_1 \in \{0,1\}$. The structure constants $\gamma^{ij}_k$ satisfy
  the following:
\begin{enumerate}
\item $\gamma^{ij}_k = \gamma^{ji}_k$ for all $i,j,k \in [n]$.
\item $\gamma^{ij}_{\ell_0 f + k} = p^{\ell_1} \eta^{\ell_1} \gamma_k^{i_0 j_0}$ for $k \in [(e- \ell_0)f]$.
\item $\gamma^{ij}_{k} \in p^{\ell_1 + 1} \Z_p$ for all $k \in [\ell_0 f]$.
\end{enumerate} 
\end{rem}

In particular, $B(\bfY)$ is symmetric and has the following block decomposition:
\begin{equation}\label{equ:B.block}
B(\bfY) = \left( \begin{matrix} B^{(0)}(\bfY) & B^{(1)}(\bfY) & \dots
  & B^{(e-1)}(\bfY)\\B^{(1)}(\bfY) &B^{(2)}(\bfY) & \hdots &
  pB^{(e)}(\bfY)\\ \vdots & \vdots & & \vdots\\ B^{(e-1)}(\bfY) &
  pB^{(e)}(\bfY)&\dots&pB^{(2e-2)}(\bfY)\end{matrix}\right),
\end{equation}
for suitable square matrices
$B^{(\mu)}(\bfY)\in\Mat_f(\Zp[\bfY])$ of $\Zp$-linear forms, for
$\mu \in[2e-2]_0$.

By the remark after \cite[Lemma 4.9]{GSS/88}, we have that
\begin{equation} \label{equ:lie.ring.equiv}
\zeta^{\triangleleft}_{H(R)} = \zeta^{\triangleleft}_{L(R)},
\end{equation}
where the ideal zeta function on the right hand side was defined in
\eqref{equ:def.ideal.zeta}; see the discussion in \cite[Section
  1.3]{SV1/13} for more details.

It is well known that, for
all $d\in\N$, the (normal) zeta function of the free abelian pro-$p$
group $\Zp^d$ of rank $d$ is given by
\begin{equation}\label{equ:free.ab}
\zeta_{\Zp^d}(s) = \prod_{i=0}^{d-1}\zeta_p(s-i),
\end{equation}
where $\zeta_p(s)=(1-p^{-s})^{-1}$ is the Euler factor of the Riemann
zeta function $\zeta(s)$ at the prime~$p$; see, for instance,
\cite[Proposition 1.1]{GSS/88}.

\subsection{The unramified case}
First suppose that $e=1$, covering the case of finite unramified
extensions $R$ of $\Zp$.

\begin{thm} \label{thm:inert}
Let $p$ be a prime and $R$ a finite unramified extension of
$\Zp$. Then
\begin{equation} \label{equ:zeta.inert}
\zideal_{H(R)}(s) = \zeta_{\Z_p^{2n}}(s) \frac{1}{1-x_0}\sum_{I\subseteq [n-1]}\binom{n}{I}_{p^{-1}}\prod_{i\in I}\frac{x_i}{1-x_i},
\end{equation}
with numerical data $x_i = p^{(2n+i)(n-i)-(3n-i)s}$ for
$i\in[n-1]_0$.
\end{thm}

\begin{proof}
In order to keep the notation of this paper compatible
with~\cite{Voll/05}, we have labeled the numerical data in reverse
order to that of~\cite{SV1/13}.  By \cite[Corollary 3.7]{SV1/13}, we
have
\begin{equation*} 
\zideal_{H(R)}(s) = \zeta_{\Z_p^{2n}}(s)
\frac{1}{1-x_0}\sum_{I\subseteq [n-1]}\binom{n}{I}_{p^{-1}}\prod_{i\in
  I}\frac{x_{n-i}}{1-x_{n-i}}.
\end{equation*}
Define $n - I \subseteq [n-1]$ to be the set $\{ n - i \mid i \in I
\}$.  Our claim follows by the identity
\begin{equation*}
\binom{n}{n-I}_{p^{-1}} = \binom{n}{I}_{p^{-1}};
\end{equation*}
cf.~\cite[Remark
  2.13]{SV1/13}.
\end{proof}

The object of this section is to give a second proof of
Theorem~\ref{thm:inert}, based on the ideas of~\cite{Voll/05}. This
will prepare the way for arguments in the general case in the
remainder of the article.

Note that, since $e = 1$, we have $\mathcal{B} = (\beta_1, \dots,
\beta_n)$.  Hence $\mathcal{B}$ reduces modulo $\mathfrak{m} = pR$ to
an $\F_p$-basis of the residue field~$k_R$.  As in \cite{Voll/05}, we
consider the Pfaffian hypersurface $\mathfrak{P}_{H(R)} \subseteq
\mathbb{P}^{n-1}$ defined by the equation $\det(B(\mathbf{Y})) = 0$.

\begin{lemma} \label{lem:nonsing}
  Let $q = p^n$, and let $T : \F_q \to \F_p$ be a non-zero
  $\F_p$-linear map.  Let $\{ x_1, \dots, x_n \}$ be an $\F_p$-basis
  of $\F_q$.  Then the matrix $A_T = (T(x_ix_j))_{ij}\in \Mat_n(\F_p)$ is
  nonsingular.
\end{lemma}
\begin{proof}
  Given $x \in \F_q$, consider the $\F_p$-linear map $U_{T,x}: \F_q
  \to \F_p$ given by $U_{T,x}(y) = T(xy)$.  Let $\F_q^\vee =
  \mathrm{Hom}_{\F_p}(\F_q, \F_p)$ be the dual space of $\F_q$.
  Observe that the map
\begin{equation*}
\F_q \rightarrow \F_q^\vee ,\quad x \mapsto U_{T,x}
\end{equation*}
is $\F_p$-linear and injective and therefore is an isomorphism of
$\F_p$-vector spaces.  The matrix $A_T$ is just the matrix of this map
with respect to the $\Fp$-basis $\{x_1, \dots, x_n \}$ and its dual
basis $\{ x_1^\vee, \dots, x_n^\vee \}$, so the claim follows.
\end{proof}

\begin{lemma} \label{lem:eee} The Pfaffian hypersurface $\mathfrak{P}_{H(R)}$ has no $\F_p$-rational points.
\end{lemma}
\begin{proof}
  Let $\mathbf{v} = (v_1, \dots, v_n)^t \in \Z^n$ be a column vector,
  and set $q = p^n$.  Let $\overline{v_i} \in \F_p$ be the reduction
  modulo $p$ of $v_i \in \Z$.  Choose an isomorphism $k_R \simeq \F_q$
  and use it to identify these two fields.  Now let
  $\overline{\beta_i} \in \F_q$ be the reduction modulo $p$ of
  $\beta_i \in R$, for~$i \in [n]$.  Recall that $\{
  \overline{\beta_1}, \dots, \overline{\beta_n} \}$ is an $\F_p$-basis
  of $\F_q$ and consider the $\F_p$-linear map $T_{\mathbf{v}} : \F_q
  \to \F_p$ given by $T_{\mathbf{v}}(\overline{\beta_i}) =
  \overline{v_i}$.  Observe that the reduction modulo $p$ of the
  matrix ${B}(\mathbf{v})$ is just the matrix $A_{T_{\mathbf{v}}}$
  defined in the statement of Lemma~\ref{lem:nonsing}.  The conclusion
  of that lemma then implies that $\det({B}(\overline{\mathbf{v}})) =
  0$ only if $\overline{\mathbf{v}} = 0$.
\end{proof}

In the notation of \cite{Voll/05}, Lemma~\ref{lem:eee} states that
$n_{\mathfrak{P}_{H(R)}}(p) = 0$ if $R$ is unramified.  Furthermore, it
implies that the Pfaffian hypersurface has no points defined
over~$\Q$.  Therefore it is vacuously smooth and has no lines.  Hence,
in the notation of \cite[Theorem~3]{Voll/05},
$\zeta^\triangleleft_{H(R),p}(s) = W_0(p,p^{-s})$, where
$W_0(p,p^{-s})$ is implicitly computed
in~\cite[Section~4.2.1]{Voll/05}. It is easily seen to match the
formula given in \eqref{equ:zeta.inert}. This concludes the second
proof of Theorem~\ref{thm:inert}.

\subsection{The general case}
We start off by describing the elementary divisors of matrices of the
form $B(\balpha)\in\Mat_n(\Zp)$, where $B(\bfY)\in\Mat_n(\Zp[\bfY])$
is defined following~\eqref{equ:mc} and $\balpha\in\Zp^n\setminus
p\Zp$. Recall the block decomposition of $B(\bfY)$ given
in~\eqref{equ:B.block}.

Given a real number $x$, we denote by $\lceil x \rceil$ the smallest
integer greater than or equal to $x$, and by $\lfloor x \rfloor$ the
largest integer less than or equal to $x$. 

\begin{dfn} \label{def:alpha.ceiling} Given $\balpha = (\alpha_1,
  \dots, \alpha_n) \in \Z_p^n \setminus p \Z_p^n$, put $\mu(\balpha) =
  \max \{ i \in [n] \mid \val(\alpha_i) = 0 \} \in [e]$ and define $\ceilalpha
  = \lceil \frac{\mu(\balpha)}{f} \rceil$.
\end{dfn}

\begin{lem}\label{lem:invert.new}
  Let $\balpha\in \Z_p^n \setminus p \Z_p^n$ with
  $\ceilalpha = m$.  Then
\begin{enumerate}
\item $B^{(m-1)}(\balpha)\in\GL_f(\Zp)$ and $B^{(\mu)}(\balpha)\in
  p\Mat_f(\Zp)$ for all $\mu \in [m, e-1]$.
\item $B^{(m+e-1)}(\balpha)\in\GL_f(\Zp)$ and $B^{(\mu)}(\balpha)\in p\Mat_f(\Zp)$ for all $\mu \in [m + e, 2e - 2]$.
\end{enumerate}
\end{lem}

\begin{proof}  Let $\mu \in [e-1]$ and let $\bar{A}^{(\mu)} \in \Mat_f(\F_p)$ be the reduction modulo $p$ of $B^{(\mu)}(\balpha)$.
  From Remark~\ref{rem:structure.constants}(3) it follows
  that $\bar{A}^{(\mu)}_{i, j} = \sum_{\ell =
    1}^{(e-\mu)f} \overline{\gamma}^{ij}_\ell \overline{\alpha}_{\mu f
    + \ell}$ for all $i,j \in [f]$, where the overline denotes reduction modulo~$p$ and the
  $\gamma^{ij}_\ell$ are as defined immediately following
  \eqref{equ:mc}.  Our assumption on $\balpha$ immediately implies the
  second part of (1), whereas if $\mu = m - 1$ then we obtain
  $\bar{A}^{(m-1)}_{i,j} = \sum_{\ell = 1}^{f}
  \overline{\gamma}_{\ell}^{ij} \overline{\alpha}_{(m-1)f + \ell}$. We
  would like to prove that $\bar{A}^{(m-1)}$ is invertible.
    
  Since $\overline{\beta}_i \overline{\beta}_j = \sum_{\ell = 1}^f
  \overline{\gamma}^{ij}_\ell \overline{\beta}_\ell$ for all $i,j \in
           [f]$, we find that $\bar{A}^{(m-1)} = A_T$, in the notation
           of Lemma~\ref{lem:nonsing}, where $T : k_R \to \F_p$ is the
           non-zero $\F_p$-linear operator given by
           $T(\overline{\beta}_i) = \overline{\alpha}_{(m-1)f + i}$
           for all $i \in [f]$. Lemma \ref{lem:nonsing} thus implies
           that $\bar{A}^{(m-1)}$ is non-singular. This establishes
           the firt part of claim~(1).
  
The second part of (2) follows similarly from Remark~\ref{rem:structure.constants}
and the hypothesis on $\balpha$.  To establish the first part of (2), we let
$\psi: k_R \to k_R$ denote the $\F_p$-linear isomorphism corresponding to
multiplication by $\overline{\pi^e / p} \in k_R^\times$ and set
$(\overline{\alpha}^\prime_1, \dots, \overline{\alpha}^\prime_f) =
\psi(\overline{\alpha}_{(m-1)f + 1}, \dots, \overline{\alpha}_{mf})$.
It follows similarly to the previous case that $\bar{A}^{(m+e-1)} =
A_{T^\prime}$, where $T^\prime(\overline{\beta}_i) =
\overline{\alpha}_i^\prime$ for all $i \in [f]$.  Thus we again have
$\bar{A}^{(m+e-1)} \in \mathrm{GL}_f(\F_p)$ by Lemma
\ref{lem:nonsing}.
\end{proof}

Write $I_\ell$ for the $\ell \times \ell$ identity matrix.  For $m \in
[e]$ we define
\begin{equation*}
J_m = \left( \begin{array}{cc} 0 & I_{mf} \\ p^{-1} I_{(e-m)f} &
  0 \end{array} \right) \in \Mat_n(\Q_p).
\end{equation*}

\begin{cor} \label{lem:balpha}
Let $m \in [e]$, and let $\balpha = (\alpha_1, \dots, \alpha_n) \in
\Z_p^n \setminus p \Z_p^n$ be a vector such that $\ceilalpha = m$.
Then $B(\balpha) J_m \in \mathrm{GL}_n (\Z_p)$.
\end{cor}

\begin{proof}
  This is immediate from Lemma~\ref{lem:invert.new} and the block
  decomposition of~\eqref{equ:B.block}.
\end{proof}

We now state the main result of this article. Recall the statistics $\len$,
$\len^{[n-2]}$, and $\Des$ on the Coxeter group~$S_n$ that were defined
in Section~\ref{subsec:cox}.

\begin{thm} \label{thm:main} Let $R$ be a finite extension of
  $\Zp$ with inertia degree $f$ and ramification index~$e$. Set
  $n=ef$. Then
  \begin{equation} \label{equ:main}
\zeta_{H(R)}^{\triangleleft}(s) =
\zeta_{\Z_p^{2n}}(s) \frac{\sum_{w \in S_n} p^{- \len (w) + 2
    f\left\lfloor\frac{\len^{[n-2]}(w)}{f}\right\rfloor s} \prod_{j \in
    \mathrm{Des}(w)} x_j }{\prod_{i = 0}^{n-1} (1 - x_i)},
\end{equation}
with numerical data $x_i = p^{(2n+i)(n-i) - (3n-i)s}$ for
$i\in[n-1]_0$.
\end{thm}

\begin{proof}
  We saw in \eqref{equ:lie.ring.equiv} that
  $\zeta_{H(R)}^{\triangleleft}(s) = \zeta_{L(R)}^{\triangleleft}(s)$,
  where $L(R)$ is the Heisenberg Lie ring over~$R$; cf.~Section
  \ref{sec:liering}.  The abelianization $\overline{L(R)}$ and the
  derived subring $L(R)^\prime$ are free $\Z_p$-modules of rank $2n$
  and $n$, respectively.  By~\cite[Lemma 1]{Voll/05}, which is
  essentially~\cite[Lemma 6.1]{GSS/88}, we have
\begin{equation} \label{equ:gafa.paper.eq}
 \zideal_{L(R)}(s) = \zeta_{\Z_p^{2n}}(s) \zeta_p(3ns - 2n^2)
 \sum_{\Lambda^\prime \leq L(R)^\prime \atop \Lambda^\prime \,
   \mathrm{maximal}} | L(R)^\prime : \Lambda^\prime |^{2n-s} | L(R) :
 X(\Lambda^\prime) |^{-s},
 \end{equation}
 where, for every finite-index sublattice $\Lambda^\prime \leq
 L(R)^\prime$, we define $X(\Lambda^\prime)$ to be the sublattice of
 $L(R)$ such that $X(\Lambda^\prime) / \Lambda^\prime$ is the center
 of $L(R) / \Lambda^\prime$.  Note that $ \zeta_p(3ns - 2n^2) =
 \frac{1}{1 - x_0}$.

 Let $\Lambda' \leq L(R)' \simeq \Z_p^n$ be a maximal sublattice of
 finite index, of type $\nu(\Lambda')=(I,\bfr_I)$, where
 $I=\{i_1,\dots,i_\ell\}_<\subseteq[n-1]$ and
 $\bfr_I=(r_{i_1},\dots,r_{i_\ell})\in\N^\ell$; cf.\
 Section~\ref{subsec:lattices}. We write $i$ for~$i_\ell$.  As in
 Section \ref{subsec:lattices}, we identify $\Lambda^\prime$ with a
 coset $\boldsymbol{\alpha} \Gamma_{(I, \bfr_I)}$, where $\boldsymbol{\alpha} \in
 \mathrm{SL}_n(\Z_p)$.  For $j \in [n]$, let $\alpha^j$ denote the
 $j$-th column vector of~$\boldsymbol{\alpha}$.  Recalling
 Definition~\eqref{def:alpha.ceiling} we set
\begin{equation} \label{def:kappa}
\kappa(\Lambda') := e-\max\{\lceil \alpha^j \rceil \mid n-i < j \leq
n\}\in [e-1]_0.
\end{equation}
An informal description of $\kappa(\Lambda')$ is as follows. Consider
the reduction modulo $p$ of the $n\times (n-i)$ matrix composed of the
last $n-i$ columns of $\boldsymbol{\alpha}$. Then $\kappa(\Lambda')=\kappa$ if and
only if the last $\kappa f$ rows of this matrix are zero, but the
$(\kappa+1)$-st block of $f$ rows from the bottom contains a nonzero
element.

The most mysterious ingredient of \eqref{equ:gafa.paper.eq} is the
quantity $|L(R) : X(\Lambda^\prime)|$, which we will now compute.

\begin{lemma} \label{lem:xlambda}
Let $\Lambda^\prime \leq L(R)^\prime$ be a maximal sublattice of type
$\nu(\Lambda^\prime) = (I, \bfr_I)$.  Then
$$ |L(R) : X(\Lambda^\prime)| = p^{2 \left( n {\sum_{\iota \in I}
    r_\iota} - \kappa(\Lambda^\prime) f \right)}.$$
\end{lemma}
\begin{proof}
  By \cite[Theorem~6]{Voll/05}, $|L(R) : X(\Lambda^\prime)|$ is equal
  to the index in $\overline{L(R)} \cong \Zp^{2n}$ of the sublattice
  of simultaneous solutions to the following system of linear
  congruences:
\begin{equation}\label{equ:cong.ramified}
  \olbfg M(\alpha^j) \equiv 0 \bmod \left(p^\nu\right)_j \textrm{ for
  }j\in[n]
\end{equation}
in variables $\olbfg = (g_1,\dots,g_{2n})$.  Here $M(\alpha^j)$ is the
commutator matrix $M(\mathbf{Y})$ of \eqref{equ:mc} evaluated
at~$\alpha^j$.  It is clear from \eqref{equ:mc} that the index of the
solution sublattice of \eqref{equ:cong.ramified} in $\overline{L(R)}$
is the square of the index in $\Z_p^n$ of the solution sublattice of
the system
\begin{equation} \label{equ:cong.simplified}
  \bfh B(\alpha^j) \equiv 0 \bmod \left(p^\nu\right)_j \textrm{ for
  }j\in[n],
\end{equation}
where $\mathbf{h} = (h_1, \dots, h_n) \in \Z_p^n$.  By Lemma
\ref{lem:balpha}, each matrix $B(\alpha^j)\in\Mat_{n}(\Zp)$ becomes
invertible after all the entries in its last $(e- \lceil \alpha^j
\rceil )f$ columns have been divided by $p$.  Therefore, $\mathbf{h} =
(h_t)\in\Zp^{n}$ is a solution to \eqref{equ:cong.simplified} if and
only if, for all $j \in [n]$,
\begin{alignat*}{2}
h_t & \equiv 0 \textrm{ mod } (p^\nu)_j& \quad\textrm{ if } t &\leq \lceil
\alpha^j \rceil f, \\ p h_t & \equiv 0 \textrm{ mod } (p^\nu)_j&\quad
\textrm{ if } t &> \lceil \alpha^j \rceil f.
\end{alignat*}
It follows that the congruences where $(p^\nu)_j$ is maximal, namely
those with $j \in [i+1,n]$, dominate all the others.
Hence, $\mathbf{h}\in\Zp^n$ is a solution to
\eqref{equ:cong.simplified} if and only if
\begin{alignat*}{2}
h_t & \equiv 0 \textrm{ mod } p^{\sum_{\iota \in I} r_\iota} &
\quad\textrm{ if } t &\leq (e - \kappa(\Lambda^\prime)) f, \\ h_t &
\equiv 0 \textrm{ mod } p^{\left( \sum_{\iota \in I} r_\iota \right) -
  1} & \quad\textrm{ if } t & > (e - \kappa(\Lambda^\prime)) f.
\end{alignat*}
Recalling that $n = ef$, it follows that the index in
$\overline{L(R)}$ of the sublattice of simultaneous solutions to the
congruences \eqref{equ:cong.ramified} is the quantity in the statement
of the lemma.
\end{proof}

It is obvious from the definition of the type of a lattice that, if
$\Lambda^\prime \leq L(R)^\prime$ is a sublattice of type $(I,
\bfr_I)$, then
\begin{equation} \label{equ:type.index}
| L(R)^\prime : \Lambda^\prime | = p^{\sum_{\iota \in I} (n - \iota) r_\iota}.
\end{equation}
Given $\kappa \in [e-1]_0$ and a type $(I, \bfr_I)$, define
\begin{equation*}
\mathcal{N}_{(I, \bfr_I)}^\kappa = \#\{\Lambda'\leq L(R)'\mid \nu(\Lambda') =
(I, \bfr_I), \kappa(\Lambda')= \kappa\}.
\end{equation*}
It follows from~\eqref{equ:gafa.paper.eq},~\eqref{equ:type.index}, and Lemma~\ref{lem:xlambda}
that
\begin{equation} \label{equ:i.kappa}
\zeta^\triangleleft_{L(R)}(s) = \frac{\zeta_{\Z_p^{2n}}(s)}{1 - x_0}
\sum_{I \subseteq [n-1]} \sum_{\bfr_I \in \N^{|I|}} \sum_{\kappa =
  0}^{e-1} \mathcal{N}_{(I, \bfr_I)}^\kappa p^{\left( \sum_{\iota \in
    I} (n - \iota) r_\iota \right)(2n - s) - 2s\left( n \sum_{\iota
    \in I}r_\iota - \kappa f \right)}.
\end{equation}

As preparation for computing the numbers $\mathcal{N}_{(I,
  \bfr_I)}^\kappa$, we fix a type $(I, \bfr_I)$ as above and consider
the surjective map
\begin{align*}\label{map:phi}
  \phi_{(I, \bfr_I)} :\{\Lambda' \mid \nu(\Lambda') = (I,\mathbf{r}_I)\} &\rarr
  \Gr(n,n-i;\Fp)\\
  \boldsymbol{\alpha} \Gamma_{(I, \bfr_I)} &\mapsto \langle \ol{\alpha^j} \mid i < j
  \leq n\rangle_{\Fp}.
\end{align*}
As before, we identify lattices of type $(I, \bfr_I)$ with cosets
$\boldsymbol{\alpha} \Gamma_{(I, \bfr_I)}$ for $\boldsymbol{\alpha}\in\Gamma$.  Informally,
$\phi(\boldsymbol{\alpha}\Gamma_{(I, \bfr_I)})$ is the subspace of $\Fp^n$ spanned
by the reduction modulo $p$ of the last $n-i$ columns of the
matrix~$\boldsymbol{\alpha}\in\Gamma = \SL_n(\Zp)$.

\begin{lem}
The fibres of $\phi_{(I, \bfr_I)}$ all have the same cardinality
\begin{equation}\label{equ:fibre.sizes.phi}
p^{-i(n-i)} \binom{i}{I \setminus \{ i \} }_{p^{-1}} p^{\sum_{\iota\in I}
  r_\iota\iota(n-\iota)}.
\end{equation}
\end{lem}

\begin{proof}
  The map $\phi_{(I, \bfr_I)}$ is just the natural surjection
\begin{equation*}\Gamma/\Gamma_{(I,\mathbf{r}_I)} \rarr \Gamma/\Gamma_{(\{i\},1)}, \quad\boldsymbol{\alpha} \Gamma_{(I,\mathbf{r}_I)} \mapsto \boldsymbol{\alpha}\Gamma_{(\{i\},1)}.
\end{equation*}
Of course $|\Gamma:\Gamma_{(I,\mathbf{r}_I)}| =
|\Gamma:\Gamma_{(\{i\},1)}||\Gamma_{(\{i\},1)}:\Gamma_{(I,\mathbf{r}_I)}|$. By
\eqref{equ:num.latt} we have
\begin{equation*}
|\Gamma:\Gamma_{(I,\mathbf{r}_I)}| = \binom{n}{I}_{p^{-1}}
p^{\sum_{\iota\in I}r_\iota\iota(n-\iota)} =
\binom{n}{n-i}_{p^{-1}}p^{r_{i}i(n-i)} \cdot
\binom{i}{I \setminus \{ i \} }_{p^{-1}}p^{\sum_{\iota\in I\setminus
    \{i\}}r_\iota\iota(n-\iota)}.
\end{equation*}
Formula~\eqref{equ:fibre.sizes.phi} for the index
$|\Gamma_{(\{i\},1)}:\Gamma_{(I,\mathbf{r}_I)}|$ follows now
from~\eqref{equ:grass}, as
\begin{equation*}
|\Gamma:\Gamma_{(\{i\},1)}| = \#\Gr(n,n-i;\Fp) = \binom{n}{n-i}_{p^{-1}}p^{i(n-i)}.
\qedhere
\end{equation*}
\end{proof}

Consider the following filtration on $\Gr(n,n-i;\Fp)$. Let
$(\epsi_1,\dots,\epsi_n)$ denote the standard $\Fp$-basis of $V = \Fp^n$,
and consider the flag $\left(V_d\right)_{d=0}^e = \left(\langle
\epsi_1,\dots,\epsi_{fd}\rangle_{\Fp}\right)_{d=0}^e$. Define
\begin{equation*}
\psi:\Gr(n,n-i;\Fp) \rarr [e],\quad
W \mapsto \min\{d \mid W\subseteq V_d\}.
\end{equation*}
One verifies easily that the fibres of $\psi$ are unions of Schubert
cells.  Indeed, if $\lambda \in [e]$ and $W \in \Gr(n,n-i;\F_p)$, then
$\psi(W) = \lambda$ if and only if the bottom $(e - \lambda)f$ rows of
the matrix of $W$ (cf.~Section~\ref{subsec:grass}) consist of zeroes,
whereas the previous block of $f$ rows does contain a non-zero matrix
element.  It is clear from the discussion in
Section~\ref{subsec:grass} that the lowest-positioned non-zero element
in the matrix of $W$ is a $1$ in the $(w(n), n - i)$ position, where
$w \in S_n$ is such that $W \in \Omega_w(\Fp)$.  In other words,
$\psi(W) = \lambda$ if and only if $w(n) \in [(\lambda - 1)f+1,
\lambda f]$.  Recalling that $\len^{[n-2]}(w) = n - w(n)$, this
condition is clearly equivalent to
$\left\lfloor\frac{\len^{[n-2]}(w)}{f}\right\rfloor = e-\lambda$.
Therefore, for every $\lambda \in [e]$,

\begin{equation}\label{equ:fibre.sizes.psi}
  \psi^{-1}(\lambda)= \bigcup_{w\in S_n,\; \Des(w)
    \subseteq\{i\} \atop \left\lfloor\frac{\len^{[n-2]}(w)}{f}\right\rfloor =
    e-\lambda}\Omega_w(\Fp).
\end{equation} 

Now consider the composition
\begin{equation*}
  \psi\circ\phi_{(I, \bfr_I)}:\{\Lambda'\leq L(R)'\mid \nu(\Lambda') = (I, \mathbf{r}_I)\} \rarr [e],\quad
  \boldsymbol{\alpha} \Gamma_{\nu} \mapsto \psi\left( \langle \ol{\alpha^j} \mid
    i < j \leq n\rangle_{\Fp}\right).
\end{equation*}
From the definition of $\kappa(\Lambda^\prime)$ in~\eqref{def:kappa}
it is evident for all $\kappa \in [e-1]_0$ that $\kappa(\Lambda') =
\kappa$ if and only if $\psi\circ\phi_{(I, \bfr_I)}(\Lambda') = e -
\kappa.$ Thus, by \eqref{equ:fibre.sizes.phi},
\eqref{equ:fibre.sizes.psi}, and the fact that
$|\Omega_w(\Fp)|=p^{i(n-i)-\len(w)}$ for all $w\in S_n$
(cf.~Section~\ref{subsec:grass}), we obtain
\begin{align} \label{equ:niri}
\mathcal{N}_{(I, \bfr_I)}^\kappa &= {\#\{\Lambda'\leq L(R)'\mid \nu(\Lambda') =
  (I, \mathrm{r}_I), \kappa(\Lambda')= \kappa\}}\\ &=\nonumber
\#\{\Lambda'\leq L(R)'\mid \nu(\Lambda') = (I,\bfr_I), \psi(\phi_{(I,
  \bfr_I)}(\Lambda'))= e - \kappa\} \\ &= \nonumber \left(
\sum_{w\in S_n,\; \Des(w)\subseteq\{i\},\;
  \left\lfloor\frac{\len^{[n-2]}(w)}{f}\right\rfloor =
  \kappa}p^{-\len(w)}\right) \binom{i}{I \setminus \{ i \}}_{p^{-1}} p^{\sum_{\iota\in
    I} r_{\iota}\iota(n-\iota)} \\ &=\nonumber \alpha^\kappa_I\prod_{\iota\in I}
p^{r_{\iota}\iota(n-\iota)},
\end{align}
where we set
\begin{equation} \label{equ:alphakappadef}
\alpha_I^\kappa = \left( \sum_{w\in
      S_n,\; \Des(w)\subseteq\{i\},\;
      \left\lfloor\frac{\len^{[n-2]}(w)}{f}\right\rfloor =
      \kappa}p^{-\len(w)} \right) \binom{i}{I \setminus \{ i \} }_{p^{-1}}.
\end{equation}

\begin{lemma}
Let $\alpha_I^\kappa$ be the quantity defined in~\eqref{equ:alphakappadef}.  Then
\begin{equation*}
\alpha^\kappa_I = \sum_{w \in S_n, \mathrm{Des}(w) \subseteq I \atop  \left\lfloor\frac{\len^{[n-2]}(w)}{f}\right\rfloor = \kappa} p^{- \len(w)}.
\end{equation*}
\end{lemma}
\begin{proof}
  It is easy to see that an element $w \in S_n$ is the unique element
  of shortest length in its coset $w W_{[i-1]}$ if and only if
  $\mathrm{Des}(w) \cap [i-1] = \varnothing$.  Writing an arbitrary
  element $w \in S_n$ in the form $w = w^{[i-1]} w_{[i-1]}$ as in
  Section~\ref{subsec:cox}, we find that $w^{[i-1]}$ is the unique
  element such that $\mathrm{Des}(w^{[i-1]}) \cap [i-1] = \varnothing$
  and $w^{[i-1]}(j) = w(j)$ for all $j > i$.  In particular,
  $\mathrm{Des}(w^{[i-1]}) \cap [i+1,n] = \mathrm{Des}(w) \cap
  [i+1,n]$, whereas $\mathrm{Des}(w_{[i-1]}) \cap [i-1] =
  \mathrm{Des}(w) \cap [i-1]$.  It follows that the elements $w \in
  S_n$ satisfying $\mathrm{Des}(w) \subseteq I$ are precisely those
  for which $\mathrm{Des}(w^{[i-1]}) \subseteq \{ i \}$ and
  $\mathrm{Des}(w_{[i-1]}) \subseteq I \setminus \{ i \}$.  Finally,
  it is clear that $w(n) = w^{[i-1]}(n)$ and hence $\len^{[n-2]}(w) =
  \len^{[n-2]}(w^{[i-1]})$.  Since $\len(w) = \len(w^{[i-1]}) +
  \len(w_{[i-1]})$ and
\begin{equation*}
\binom{i}{I \setminus \{ i \}}_{p^{-1}} = \sum_{w \in W_{[i-1]} \atop \mathrm{Des}(w) \subseteq I \setminus \{ i \} } p^{- \len (w)},
\end{equation*}
the desired equality follows.
\end{proof}

Finally, we have all the ingredients necessary to compute
$\zeta^{\triangleleft}_{H(R)}(s)$ and finish the proof of
Theorem~\ref{thm:main}. Indeed, a simple calculation
using~\eqref{equ:niri} shows that
\begin{multline*}
 \sum_{\bfr_I \in \N^{| I |}} \mathcal{N}_{(I, \bfr_I)}^\kappa
 p^{\left( \sum_{\iota \in I} (n - \iota) r_\iota \right)(2n - s) -
   2s\left( n \sum_{\iota \in I}r_\iota -\kappa f \right)}
 =\\ \alpha_I^\kappa p^{2s\kappa f} \sum_{\bfr_I \in \N^{| I |}}
 \prod_{\iota \in I} \left( p^{(2n + \iota)(n - \iota) - (3n -
   \iota)s} \right)^{r_\iota} = \alpha_I^\kappa p^{2s\kappa f}
 \prod_{\iota \in I} \frac{x_\iota}{1 - x_\iota},
\end{multline*}
where $x_\iota = p^{(2n + \iota)(n - \iota) - (3n - \iota)s}$ for
$\iota\in I \subseteq[n-1]$ as in the statement of
Theorem~\ref{thm:main}.  By~\eqref{equ:i.kappa} this implies
\begin{equation*}
{\zeta^{\triangleleft}_{H(R)}(s) \frac{1-x_0}{\zeta_{\Zp^{2n}}(s)}} =
\sum_{I \subseteq [n-1]} \sum_{\kappa = 0}^{e-1} \alpha_I^\kappa
p^{2\kappa f s} \prod_{\iota \in I} \frac{x_\iota}{1 - x_\iota}.
\end{equation*}
Bringing the right hand side to a common denominator, we get
\begin{equation*}
\frac{\sum_{\kappa = 0}^{e-1} \sum_{I \subseteq [n-1]}  \beta^\kappa_I  p^{2 \kappa f
    s}\prod_{\iota \in I} x_\iota}{\prod_{i = 1}^{n-1}
  (1 - x_i)},
\end{equation*}
where
\begin{equation*}
\beta^\kappa_I = \sum_{J \subseteq I} (-1)^{| I | - | J |} \alpha^\kappa_J = \sum_{w \in S_n, \,\Des(w) = I \atop \left\lfloor\frac{\len^{[n-2]}(w)}{f}\right\rfloor = \kappa} p^{- \len(w)},
\end{equation*}
the last equality following from a simple inclusion-exclusion
argument; cf.~\cite[(1.34)]{Stanley/12}.  Hence,
\begin{align*}
 \zeta^{\triangleleft}_{H(R)}(s) \frac{\prod_{i = 0}^{n-1}
   (1-x_i)}{\zeta_{\Zp^{2n}}(s)} &= \sum_{I \subseteq [n-1]}
 \sum_{\kappa = 0}^{e-1} p^{2 \kappa f s} \left( \sum_{w \in S_n, \,
   \Des(w) = I \atop \left\lfloor
   \frac{\len^{[n-2]}(w)}{f}\right\rfloor = \kappa} p^{- \len(w)}
 \right) \prod_{\iota \in I} x_\iota \\ &= \sum_{w \in S_n} p^{-
   \len(w) + 2f \left\lfloor \frac{\len^{[n-2]}(w)}{f}\right\rfloor s}
 \prod_{j \in \Des(w)} x_j,
\end{align*}
as claimed.
\end{proof}

\begin{rem}
Theorem~\ref{thm:inert} is indeed a special case of
Theorem~\ref{thm:main}. If $e = 1$ and $f=n$, then $\len^{[n-2]}(w)
\leq n - 1 < f$ for all $w \in S_n$, and hence
$\left\lfloor\frac{\len^{[n-2]}(w)}{f}\right\rfloor=0$ for all $w\in S_n$.  One
verifies easily that
$$\frac{\sum_{w \in S_n} p^{- \len (w)} \prod_{j \in \mathrm{Des}(w)}
  x_j }{\prod_{i = 0}^{n-1} (1 - x_i)} =
\frac{1}{1-x_0}\sum_{I\subseteq [n-1]}\binom{n}{I}_{p^{-1}}\prod_{i\in
  I}\frac{x_i}{1-x_i},$$ by bringing the right hand side to a common
denominator as in \cite[Section 4.1]{Voll/05} and using
\eqref{equ:des}.
\end{rem}

\begin{cor}\label{cor:funeq}
Let $R$ be a finite extension of $\Z_p$ with inertia degree $f$ and ramification index~$e$. Set
  $n=ef$.  Then $\zeta_{H(R)}^{\triangleleft}(s)$ satisfies the following functional equation:
\begin{equation*} 
\zeta_{H(R)}^{\triangleleft}(s)|_{p \rarr p^{-1}} = (-1)^{3n}
p^{\binom{3n}{2} - (5n + 2 (e-1)f)s}
\zeta_{H(R)}^{\triangleleft}(s).
\end{equation*}
\end{cor}

\begin{proof}
  Recall from~\eqref{equ:len.par} that, for all $w \in S_n$, we have
  $\len^{[n-2]}(w_0w) = \len^{[n-2]}(w_0) - \len^{[n-2]}(w)$, where
  $w_0 \in S_n$ is longest element.  The key observation is that the
  fractional part of $\frac{\len^{[n-2]}(w_0)}{f} = \frac{n-1}{f} = e
  - \frac{1}{f}$ is the largest possible.  Hence, for all $w\in S_n$,
\begin{equation*}
\left\lfloor \frac{\len^{[n-2]}(w_0w)}{f} \right\rfloor = \left\lfloor
\frac{\len^{[n-2]}(w_0)}{f}\right\rfloor -
\left\lfloor\frac{\len^{[n-2]}(w)}{f}\right\rfloor = (e-1)
-\left\lfloor\frac{\len^{[n-2]}(w)}{f}\right\rfloor.
\end{equation*}
Using \eqref{equ:free.ab},
\eqref{equ:des.w0}, and \eqref{equ:des} it then follows that
\begin{align*}
  \lefteqn{\zideal_{H(R)}(s)\vert_{p \rarr p^{-1}}=\zeta_{\Zp^{2n}}(s)
    \vert_{p \rarr p^{-1}} \frac{\sum_{w\in S_n}p^{\len(w) -
        2f\left\lfloor\frac{\len^{[n-2]}(w)}{f}\right\rfloor s}
      \prod_{j\in \Des(w)}
      x_j^{-1}}{\prod_{i=0}^{n-1}(1-x_i^{-1})}}\\&=(-1)^{3n}p^{\binom{2n}{2}-2ns}\zeta_{\Zp^{2n}}(s)
  x_0\frac{\sum_{w\in S_n}p^{\len(w) -
      2f\left\lfloor\frac{\len^{[n-2]}(w)}{f}\right\rfloor s}
    \prod_{j\in [n-1]\setminus
      \Des(w)}x_j}{\prod_{i=0}^{n-1}(1-x_i)}\\&=(-1)^{3n}p^{\binom{3n}{2}-(5n+2(e-1)f)s}\zeta_{\Zp^{2n}}(s)
  \frac{\sum_{w_0w\in S_n}p^{-\len(w_0w) +
      2f\left\lfloor\frac{\len^{[n-2]}(w_0w)}{f}\right\rfloor s}
    \prod_{j\in \Des(w_0w)}x_j}{\prod_{i=0}^{n-1}(1-x_i)}
  \\&=(-1)^{3n} p^{\binom{3n}{2} - (5n + 2(e-1)f)s}
  \zeta_{H(R)}^{\triangleleft}(s),
\end{align*}
as claimed.
\end{proof}

\subsection{An alternative formulation of the main result}
We now prove an alternative formula for $\zideal_{H(R)}(s)$ to that of
Theorem~\ref{thm:main} by showing that, in general, the fraction on
the right hand side of~\eqref{equ:main} admits some cancellation.

Consider the $n$-cycle $c = (1 \mbox{ } 2 \cdots n)\in S_n$. For
$i\in[n-1]_0$, let $x_i$ be as in Theorem~\ref{thm:main}.
\begin{lemma}\label{lem:shift}
  Let $w \in S_n$, and let $m \in [n-1]_0$ be such that $w(1)
  \leq n - m$.  Then
\begin{equation*} 
  p^{- \len (c^m w) + 2 \len^{[n-2]} (c^m w) s} \prod_{j \in \mathrm{Des}(c^m w)} x_j  = (p^{2n - 3s})^m p^{- \len (w) + 2 \len^{[n-2]} (w) s} \prod_{j \in \mathrm{Des}(w)} x_j.
\end{equation*}
\end{lemma}
\begin{proof}
  It suffices to prove the statement for $m = 1$; the general case
  clearly follows from iterated application of this result.  So let $w
  \in S_n$ and set $j = w^{-1}(n)$. Recall that $\len^{[n-2]}(w) = n -
  w(n)$ and observe that $\len(c w) - \len (w) = 2j - n - 1$.
  Moreover, we observe that $ \mathrm{Des}(c w) =
  \left(\mathrm{Des}(w) \cup \{ j - 1 \}\right) \setminus \{ j \}$.

  If $j<n$, then $w(n) < n$ and so $w(n) - cw(n) = -1$. If $j = n$,
  then $w(n) - cw(n) = n - 1$. In either case we obtain, by setting $x_n:= 1$, for $j\in[n]$, that
\begin{equation*}
  \frac{ p^{- \len (cw) + 2 \len^{[n-2]} (cw) s} \prod_{j \in
      \mathrm{Des}(cw)} x_j}{p^{- \len (w) + 2 \len^{[n-2]} (w) s}
    \prod_{j \in \mathrm{Des}(w)} x_j} =
  p^{n - 2j + 1 + 2(w(n) - cw(n))s} \frac{x_{j-1}}{x_j} = p^{2n - 3s}.
\end{equation*}

\end{proof}

\begin{lemma}\label{lem:aux}
  Suppose that $w \in S_n$ and $f\in\N$ satisfies $w(1) \leq f$.  Then
  for any $m \leq \lfloor \frac{n-f}{f} \rfloor$ the following holds:
\begin{multline*}
p^{- \len (c^{mf} w) + 2 f\left\lfloor\frac{\len^{[n-2]}(c^{mf}
    w)}{f}\right\rfloor s} \prod_{j \in \mathrm{Des}(c^{mf} w)} x_j
=\\ (p^{2n - 3s})^{mf} p^{- \len (w) + 2
  f\left\lfloor\frac{\len^{[n-2]}(w)}{f}\right\rfloor s} \prod_{j \in
  \mathrm{Des}(w)} x_j.
\end{multline*}
\end{lemma}
\begin{proof}
  Let $w$ and $m$ be as in the lemma. Then $\len^{[n-2]}(c^{mf}w) =
  \len^{[n-2]}(w) - mf$.  Thus 
$$ f\left\lfloor\frac{\len^{[n-2]}(c^{mf} w)}{f} \right\rfloor -
  \len^{[n-2]}(c^{mf} w) = f\left\lfloor\frac{\len^{[n-2]}(w)}{f}
  \right\rfloor - \len^{[n-2]}(w) .
$$
The lemma follows immediately from this and Lemma~\ref{lem:shift}.
\end{proof}

\begin{thm}\label{thm:Snf} 
Let $R$ be a finite extension of $\Zp$ with inertia degree $f$ and
ramification index~$e$. Set $n=ef$ and $S_n^{(f)} = \left\{ w \in S_n
\mid w(1) \leq f \right\}$.  Then
\begin{equation*}
  \zeta_{H(R)}^{\triangleleft}(s) =
  \zeta_{\Z_p^{2n}}(s) \frac{\sum_{w \in S^{(f)}_n} p^{- \len
      (w) + 2 f\left\lfloor\frac{\len^{[n-2]}(w)}{f}\right\rfloor s} \prod_{j \in
      \mathrm{Des}(w)} x_j }{(1 - p^{f(2n - 3s)}) \prod_{i = 1}^{n-1} (1 - x_i)}.
\end{equation*}
\end{thm}
\begin{proof}
  This follows from Theorem~\ref{thm:main}, Lemma~\ref{lem:aux}, and
  the observations that
$$ 1 - x_0 = 1 - (p^{2n - 3s})^n = (1 - (p^{f(2n - 3s)})) \sum_{m =
  0}^{e - 1} p^{mf(2n - 3s)} $$ and that every element of $S_n$ can be
written uniquely in the form $c^{mf} w$, where $m \in [e-1]_0$
and $w \in S^{(f)}_n$.
\end{proof}

\begin{rem}
An interesting question is whether the fraction in
Theorem~\ref{thm:Snf} is always in lowest terms and admits no more
cancellation.  T.~Bauer has verified this for all pairs $(e,f)$
with~$n = ef \leq 10$.
\end{rem}

In the case that $e=1$, Theorem~\ref{thm:Snf} is exactly
Theorem~\ref{thm:inert}. In the other extreme, the case $f=1$, we
obtain an interesting corollary.  We identify $S_n^{(1)}$, the
stabilizer in $S_n$ of the letter~$1$, viz.\ the parabolic subgroup
$(S_n)_{\{s_i\mid\, 2 \leq i \leq n-1\}}$, with~$S_{n-1}$.

\begin{cor}\label{cor:tot.ram}
  Let $R$ be a totally ramified extension of~$\Zp$ of degree~$n$. Then
\begin{equation}\label{equ:tot.ram.cancel}
  \zeta_{H(R)}^{\triangleleft}(s) = \zeta_{\Z_p^{2n}}(s)
  \frac{\sum_{w \in S_{n-1}} p^{- \len(w) + 2 \len^{[n-3]}(w) s}
    \prod_{j \in \mathrm{Des}(w)} x_{j+1}}{(1 - p^{2n - 3s}) \prod_{i
      = 1}^{n-1} (1 - x_i)},
\end{equation}
with numerical data $x_i = p^{(2n+i)(n-i) - (3n-i)s}$ for $i\in[n-1]$.
\end{cor}

\begin{exm} 
We illustrate our results in the case $e=3$, $f=1$. Thus let $R$ be a
totally ramified cubic extension of $\Zp$. It is shown in
\cite[Proposition~8.15]{GSS/88} that
\begin{equation}\label{equ:cubic.ram.GSS}
\zideal_{H(R)}(s) = \frac{1 +
  p^{7-5s}}{\left(\prod_{i=0}^{5}(1-p^{i-s})\right)
  (1-p^{6-3s})(1-p^{8-7s})(1-p^{14-8s})}.
\end{equation}
Theorem~\ref{thm:main} presents this zeta function as
  \begin{equation}\label{equ:cubic.ram.SV}
\zeta_{H(R)}^{\triangleleft}(s) = \zeta_{\Z_p^{6}}(s)
\frac{\sum_{w \in S_3} p^{- \len (w) + 2 \len^{\{1\}}(w)s} \prod_{j
    \in \Des(w)} x_{j} }{\prod_{i = 0}^{2} (1 - x_i)},
\end{equation}
with numerical data
\begin{equation*}
x_0 = p^{18-9s}, \quad
x_1 = p^{14-8s},\quad
x_2 = p^{8-7s}.
\end{equation*}
The Coxeter group $S_3$ is generated by the involutions $s_1$ and
$s_2$. We tabulate the values of the functions $\Des$, $\len$, and
$\len^{\{1\}}$ on $S_3$.

\begin{center}
\begin{tabular}{c|ccc}
$w\in S_3$ & $\Des(w)$ & $\len(w)$ & $\len^{\{1\}}(w)$\\ \hline
  $1$ & $\varnothing$ & $0$ & $0$ \\ $s_1$ & $\{1\}$ & $1$ & $0$
  \\ $s_2$ & $\{2\}$ & $1$ & $1$ \\ $s_2s_1$ & $\{1\}$ & $2$ & $1$
  \\ $s_1s_2$ & $\{2\}$ & $2$ & $2$ \\ $(s_2s_1s_2 =) s_1s_2s_1$ &
  $\{1,2\}$ & $3$ & $2$
\end{tabular}
\end{center}
We deduce that 
\begin{multline*}
\sum_{w\in S_3} p^{-\len(w) +
    2\len^{\{1\}}(w)s}\prod_{j\in\Des(w)}x_j\\= 1 + p^{13-8s} +
p^{7-5s}+ p^{12-6s}+ p^{6-3s}+ p^{19-11s}= \frac{(1- p^{18-9s})(1+
  p^{7-5s})}{1-p^{6-3s}},
\end{multline*}
showing that \eqref{equ:cubic.ram.GSS} accords
with~\eqref{equ:cubic.ram.SV}. Formula \eqref{equ:cubic.ram.GSS} also
illustrates~\eqref{equ:tot.ram.cancel}, as $\langle s_2\rangle \cong
S_2$ and $$1 + p^{7-5s} = 1+ p^{-1+2s}x_2 = \sum_{w\in S_2}p^{-\len(w)
  + 2\len(w)s} \prod_{j\in\Des(w)} x_{j+1}.$$
\end{exm}

\bibliographystyle{amsplain}

\def\cprime{$'$} \def\cprime{$'$}
\providecommand{\bysame}{\leavevmode\hbox to3em{\hrulefill}\thinspace}
\providecommand{\MR}{\relax\ifhmode\unskip\space\fi MR }
\providecommand{\MRhref}[2]{%
  \href{http://www.ams.org/mathscinet-getitem?mr=#1}{#2}
}
\providecommand{\href}[2]{#2}

\end{document}